\documentclass[12pt,leqno]{article}
\usepackage{graphicx, fullpage}
\usepackage{amsmath,amssymb,amsthm,amscd, bm}
\usepackage{fancyhdr}
\usepackage[mathscr]{eucal}
\usepackage{amsfonts}
\begin{document}
\parskip=6pt
\theoremstyle{plain}
\newtheorem{prop}{Proposition}[section]
\newtheorem{lem}[prop]{Lemma}
\newtheorem{thm}[prop]{Theorem}
\newtheorem{cor}[prop]{Corollary}
\newtheorem{defn}[prop]{Definition}
\theoremstyle{definition}
\newtheorem{example}[prop]{Example}
\theoremstyle{remark}
\newtheorem{remark}[prop]{Remark}
\numberwithin{prop}{section}
\numberwithin{equation}{section}
\def\cal{\mathcal}
\newcommand{\cF}{\cal F}
\newcommand{\cA}{\cal A}
\newcommand{\cC}{\cal C}
\newcommand{\cO}{{\cal O}}
\newcommand{\cE}{{\cal E}}
\newcommand{\cU}{{\cal U}}
\newcommand{\cM}{{\cal M}}
\newcommand{\cD}{{\cal D}}
\newcommand{\cK}{{\cal K}}
\newcommand{\cZ}{{\cal Z}}
\newcommand{\bC}{\mathbb C}
\newcommand{\bP}{\mathbb P}
\newcommand{\bN}{\mathbb N}
\newcommand{\bA}{\mathbb A}
\newcommand{\bR}{\mathbb R}
\newcommand{\oP}{\overline P}
\newcommand{\oQ}{\overline Q}
\newcommand{\oR}{\overline R}
\newcommand{\oS}{\overline S}
\newcommand{\oc}{\overline c}
\newcommand{\bp}{\mathbb p}
\newcommand{\oD}{\overline D}
\newcommand{\oE}{\overline E}
\newcommand{\oC}{\overline C}
\newcommand{\of}{\overline f}
\newcommand{\ou}{\overline u}
\newcommand{\ow}{\overline w}
\newcommand{\oy}{\overline y}
\newcommand{\oz}{\overline z}
\newcommand{\hg}{\hat G}
\newcommand{\hM}{\hat M}
\newcommand{\tpr}{\widetilde {\text{pr}}}
\newcommand{\tB}{\widetilde B}
\newcommand{\tx}{\widetilde x}
\newcommand{\ty}{\widetilde y}
\newcommand{\txi}{\widetilde \xi}
\newcommand{\teta}{\widetilde \eta}
\newcommand{\tna}{\widetilde \nabla}
\newcommand{\tth}{\widetilde \theta}
\newcommand{\diml}{\text{dim}}
\newcommand{\var}{\varepsilon}
\newcommand{\End}{\text{End }}
\newcommand{\loc}{\text{loc}}
\newcommand{\lam}{\lambda}
\newcommand{\Hom}{\text{Hom}}
\newcommand{\Ker}{\text{Ker}}
\newcommand{\dist}{\text{dist}}
\newcommand{\psl}{\rm{PSL}}
\newcommand{\rk}{\roman{rk}}
\newcommand{\In}{\text{int}\,}
\newcommand{\Id}{\text{Id}}
\newcommand{\id}{\text{id}}
\renewcommand\qed{ }
\begin{titlepage}\title{ Adapted complex and involutive structures}
\author{L\'aszl\'o Lempert \thanks{Research partially  supported by NSF grant DMS 1764167.  The paper was written at the E\"otv\"os University, Budapest while I was on sabbatical leave from Purdue University. I am grateful to both institutions.}\\ Department of  Mathematics\\
Purdue University\\West Lafayette, IN
47907-2067, USA}
\end{titlepage}
\date{}
\maketitle
\abstract
We define and study complex structures and generalizations on spaces consisting of geodesics or harmonic maps that are compatible with the symmetries of these spaces. The main results are about existence and uniqueness of such structures.
\endabstract

\section{Introduction}

This paper addresses instances of the following general question: What natural structures exist on the space of solutions of a given differential equation? The simplest answer is, a vector space structure, if the differential equation is linear. This is probably also the most momentous answer, as it led to the birth of linear algebra, see \cite{Gr}. However, in this paper we will consider certain nonlinear equations and the structures sought will be geometric rather than algebraic. This line of research is over a hundred years old, with perhaps the first result going back to W\"unschmann, who in his thesis \cite{W} found a class of scalar third order ordinary differential equations on the solution space of which he could define a canonical conformal structure. Subsequently E.~Cartan and Chern further developed the subject, that is still alive, see \cite{Ca, Ch, GN}. Conformal structures also appear in Penrose's twistor theory \cite{P, AHS} on the space of global solutions of certain partial differential equations (Cauchy--Riemann equations on certain holomorphic vector bundles).

The structures in this paper are not conformal but complex structures and generalizations that are adapted in a sense to symmetries of the spaces of solutions. Adapted complex structures first appeared in a paper by Sz\H{o}ke and this author \cite{LSz1} as complex structures on the tangent bundle of a Riemannian manifold $M$ that are compatible in a certain sense with the Riemannian metric. Simultaneously Guillemin and Stenzel in \cite{GS} investigated related complex structures on the cotangent bundle of $M$; the two structures turn out to be the same if $TM$ and $T^*M$ are identified using the metric. Subsequently Bielawski, Duchamp--Kalka, Hall--Kirwin, Sz\H{o}ke, and Ursitti \cite{Bi, DK, HK1, HK2, Sz4, U} proposed generalizations that involved manifolds with connection, Finsler  metrics, sub-Riemannian manifolds and Riemannian manifolds endowed with a closed 2-form. Independently, Thiemann \cite{Th} constructed related complex structures for general analytic Hamiltonian systems.

In \cite{LSz2} with Sz\H{o}ke we proposed yet another look at adapted complex structures. The idea was first to renounce $TM$ or $T^*M$ as the manifold underlying adapted complex structures and consider instead the manifold $N$ of geodesics in the (say, complete) Riemannian manifold $M$. This $N$ can be identified with $TM$ or $T^*M$, but it has the advantage that it naturally carries an action of the group $G$ of affine transformations $\bR\to \bR$ (reparametrizations of geodesics). The second part of the idea was then to endow $G$ with the structure of a complex manifold, and search for complex structures on $N$ such that the two complex structures are 
compatible via the action. The precise meaning of this compatiblity will be explained in Section 2, in greater generality. What matters, though, is that it makes sense regardless of geodesics and reparametrizations and one can speak of adapted complex structures whenever a Lie group $G$, endowed with a complex structure, acts on a manifold $N$. Nevertheless, when $N$ is the manifold of geodesics in a Riemannian manifold, acted on by affine reparametrizations, the new notion of adapted complex structures essentially agrees with the one in \cite{LSz1}.

In this paper we introduce a generalization, so-called involutive structures adapted to group actions. Although in this context group actions are a little too restrictive for interesting applications and in the main body of the paper we work with actions of what we call Lie monoids, to explain the main notions here we confine ourselves to group actions. An involutive structure on a smooth ($=C^\infty$) manifold $N$ is a smooth subbundle $P\subset\bC TN=\bC\otimes_{\bR}TN$ that is involutive in the sense that the Lie bracket of its (smooth, local) sections is again a section (see \cite{Tr}). For example, an involutive structure with $P\oplus\oP=\bC TN$ is the same as a complex structure, $P$ corresponding to the bundle of $(1,0)$ vectors. More generally, if $P\cap\oP=(0)$, we obtain a CR structure on $N$. At the other extreme $P=\oP$, which defines a foliation of $N$, $P$ being the complexified leafwise tangent bundle. By an involutive manifold we mean a smooth manifold $N$ endowed with an involutive structure $P$, and we use the notation $(N,P)$. If $(M,Q)$ is another involutive manifold, a $C^1$ map $f:(M,Q)\to (N,P)$ is called involutive if $f_* Q\subset P$.

Suppose that on a smooth manifold $N$ a Lie group $G$ acts on the right, and that $G$ is endowed with an involutive structure $S$. 

\begin{defn}
An involutive structure on $N$ is adapted  to the action of $(G,S)$ if for every $x\in N$ the orbit map 
\[G\ni g\mapsto xg\in N\]
is involutive.

\end{defn}

Adapted involutive structures first appeared in \cite{Sz3} in the context of Riemannian manifolds. Sz\H{o}ke, among other things, 
explains that while for a symmetric space  $M$ an adapted complex structure on $TM$ may only exist in a neighborhood of the zero section, it always 
extends to all of $TM$ as an adapted involutive structure.  In \cite{Sz5} he extended this result to manifolds endowed with 
a connection, whose curvature is covariantly constant.
CR structures that can be viewed as adapted also appear in Aguilar's quotient 
construction in \cite{A}. In \cite{LSz2} with Sz\H{o}ke we introduced, already in the context of actions, adapted polarizations. These are 
involutive structures $P$ with $2\, {\text{rk}}_\bC\, P=\dim_\bR N$. Overall, adapted involutive structures seem to be a natural notion worth studying. If adapted complex structures appear in spaces of geodesics, in section 4 we show that certain spaces of harmonic maps carry adapted involutive structures. Adapted involutive structures can be useful: in \cite{L2} we base on them the proof of a regularity result in Riemannian geometry.

The main results of this paper concern existence and uniqueness of adapted involutive structures, see Sections 3, 4, and 7. Both types of results depend on assumptions that we make on $G$, its action on $N$, on the involutive structure $S$ of $G$ and even some crude assumptions on the involutive structure $P$ of $N$. Once the right assumptions are made, the existence result follows rather easily via the device of complexification. Uniqueness is more involved, and verifying the assumptions of, say, Theorem 7.3 can be challenging. We verify the assumptions in the case when $N$ is a space of geodesics, and $G$ is the affine group, in the proof of Theorem 7.5; but for the case of harmonic maps and $G=\text{PSL}_2(\bC)$ we refer to a forthcoming paper.---Sections 5, 6 develop geometric notions in involutive manifolds.

In making our definitions and assumptions we strove for a generality that includes the above two cases: $N$ the geodesics $x:[-r,r]\to M$ in a Riemannian manifold $M$, and the H\"older continuous harmonic maps $D\to M$ from the 2-disc that are close to constant. Since this latter is a Banach manifold, we allow $N$ to be such. On both spaces reparametrizations by certain affine, resp. conformal transformations act; which, however, do not form a group. For this reason in this paper we will deal with actions of monoids, rather than just groups. 

\section{Actions of Lie monoids}

Consider a finite dimensional smooth manifold $X$ and a subset $G\subset X$ that is regular closed (i.e., $G$ is the closure of its interior). By a smooth function or 
map on $G$ we mean one that is smooth on $\text{int}\,G$, and whose derivatives of all orders, computed in charts $U\subset X$, continuously extend 
from $U\cap \text{int}\,G$ to $U\cap G$. We call $TG=TX|G\to G$ the tangent bundle of $G$. Thus $TG\subset TX$ is also a regular closed 
subset. A 
smooth map $f:G\to Y$ into a smooth manifold $Y$ induces a smooth map $f_*: TG\to TY$, first on int$\,TG$ by the traditional recipe, then by 
continuous extension to all of $TG$.

\begin{defn} 
A Lie monoid is a regular closed subset $G$ of a smooth finite dimensional manifold, endowed with a smooth and associative multiplication $G\times G\ni(g,h)\mapsto gh\in G$ that is an open map, and with a two sided unit $e\in G$ for this multiplication. 
\end{defn}  

One example is, given $r>0$, affine transformations $t\mapsto a+bt$ of the real line such that $b>0$ and $|a|+br\le r$. This is 
viewed as a subset of the group of all affine transformations of $\bR$; it consists of those transformations that map $[-r,r]$ strictly increasingly into itself. Another is the set $G\subset\text {PSL}_2(\bC)$ of conformal maps of the Riemann sphere that map a fixed bounded convex domain $D\subset\bC$ into itself.

Consider next a smooth manifold $N$, that we allow to be a Banach manifold. For the foundations of infinite dimensional differential  geometry \cite{La} by Lang is a good source, whose terminology we will follow with the following exception. We call a $C^1$ map $f:M\to N$ of Banach manifolds {\sl direct} if for each $p\in M$ the subspaces Ker$f_*(p)\subset T_pM$ and Im$f_*(p)\subset T_{f(p)}N$ have closed complements. We say $f$ is immersive/submersive if Ker$f_*(p)=(0)$, resp. Im$f_*(p)=T_{f(p)}N$, while what Lang calls immersion/submersion are direct immersion/submersion in our terminology.

\begin{defn}
A right action of a Lie monoid $G$ on $N$ is a continuous map
\[N\times G\ni(x,g)\mapsto xg \in N\]
that is smooth on $N\times {\rm{int}}\, G$, and its higher derivatives along $N$ have continuous extensions to $N\times G$; furthermore
\[x(gh)=(xg)h\quad {\rm{and}}\quad xe=x\quad {\rm{for}}\quad x\in N,\, g, h\in G.\]
\end{defn} 

Given such an action we write
\begin{equation}  
xg=A_g x=\Omega_x g.
\end{equation}   
Thus $\Omega_x$ is the orbit map of $x$.---Left actions are defined analogously.

Finally we bring in involutive structures and their interplay with actions. The definition of an involutive structure $P$ on $N$ was already given in the Introduction. When $N$ is a Banach manifold, the requirement on $P$ is that it be a Banach subbundle of $\bC TN$. However, we need to specify what we mean by an involutive structure on a Lie monoid $G$ that is a regular closed subset of a manifold $X$. 

\begin{defn} 
An involutive structure on a Lie monoid $G$ is a continuous subbundle $S\subset \bC TG$ that is smooth and involutive over ${\rm{int}}\,G$.
\end{defn} 

We will call $G$ endowed with such an $S$ an involutive Lie monoid.

\begin{defn} 
Suppose that on a manifold $N$ a Lie monoid acts on the right, and $S\subset \bC TG$ is an involutive structure. An involutive structure $P$ on $N$ is adapted to the action of $(G, S)$ if for all $x\in N$ the map $\Omega_x |{\rm{int}}\,G:({\rm{int}}\,G, S)\to (N,P)$ is involutive, cf. (2.1). 
\end{defn} 

Often the existence of an adapted $P$ forces $S$ to be left invariant, i.e. such that the differential of left translation by any $g\in G$ maps $S$ into itself.

\begin{prop} 
Suppose an involutive structure $P$ on a manifold $N$ is adapted to the right action of an involutive Lie monoid $(G,S)$. If
\begin{equation}
S_g=(\Omega_{x*})^{-1} P_{xg}, \quad g\in{\rm{int}}\,G, \, x\in N
\end{equation}
(or only for all $g\in{\rm{int}}\,G$ and for some $x$ depending on $g$), then $S$ is left invariant.
\end{prop} 

Note that $P$ is adapted means the weaker condition $S_g\subset (\Omega_{x*})^{-1} P_{xg}$. 

\begin{proof}
Suppose first $g,h$, and $gh\in\text{int}\,G$, and choose $x\in N$ so that $S_{gh}=(\Omega_{x*})^{-1}P_{xgh}$. Denoting $L_g: G\to G$ left translation by $g$, we have $\Omega_xL_g=\Omega_{xg}$. Therefore
\[\Omega_{x*}(L_{g*}S_h)=(\Omega_{xg})_*S_h\subset P_{xgh},\quad{\rm{and}}\quad  L_{g*}S_h\subset\Omega_{x*}^{-1} P_{xgh}=S_{gh}.\]

Next take arbitrary $g,h\in G$. Since $G$ is regular closed and multiplication is an open map, there are 
$g_\nu,h_\nu\in\text{int}\,G  \, (\nu\in\bN)$, converging to $g, h$ such that $g_\nu h_\nu\in\text{int}\,G$. By what we have 
already proved,
$L_{g*}S_h=\lim\limits_{\nu} L_{g_{\nu}*}S_{h_\nu}\subset \lim\limits_{\nu} S_{g_{\nu}h_\nu}=S_{gh}$,
as claimed.
\end{proof}

Sometimes condition (2.2) can be checked without exactly knowing what $S$ and $P$ are. For example, if $S$ is a complex structure, $P$ is a CR structure $(S\oplus\oS=\bC TG, P\cap\oP=(0))$, and for every $g\in\text{int}\,G$ there is an $x\in N$ such that $\Omega_x$ is an immersion at $g$, then (2.2) follows. Indeed, if $\sigma'\in S_g$, $\sigma''\in\oS_g$, and $\sigma'+\sigma{''}\in \Omega_{x*}^{-1} P_{xg}$, then $\Omega_{x*}\sigma'\in P_{xg}$ implies $\Omega_{x*}\sigma{''}\in P_{xg}$; but also $\Omega_{x*}\sigma{''}\in \oP_{xg}$, so $\Omega_{x*}\sigma{''}=0$, $\sigma{''}=0$ and $\sigma'+\sigma{''}\in S_g$. In what follows, we will mostly consider left invariant involutive structures $S$ on $G$ and involutive structures adapted to actions of $(G,S)$.

\section{Existence} 
 
In this section under certain analyticity assumptions we will construct left invariant involutive structures on Lie monoids $G$ and adapted involutive structures on manifolds $N$ on which $G$ acts. We start with a general property of involutive bundles.

\begin{lem} 
Let $X, Y$ be smooth manifolds, $P\subset\bC TX$ and $Q\subset\bC TY$ smooth subbundles, and $\varphi:X\to Y$ a smooth map such that $P_x=\varphi^{-1}_* Q_{\varphi(x)}$ for $x\in X$. If $Q$ is involutive, then so is $P$. 
\end{lem} 

\begin{proof}
A bundle, say $P\subset\bC TX$, is involutive if and only if for all $x\in X$ there is a collection $A$ of smooth  1--forms defined in a neighborhood $U$ of $x$ such that
\begin{equation}  
P|U=\bigcap \{{\rm{Ker}}\, \alpha : \alpha\in A\}, \quad\text{and}\quad d\alpha|P_x=0 \quad\text{for } \alpha\in A.
\end{equation}
 This follows from the formula  $d\alpha(\xi,\eta)=\xi\alpha(\eta)-\eta\alpha(\xi)-\alpha[\xi,\eta]$ for vector fields $\xi, \eta$.

Now suppose first that $\varphi$ is a direct embedding of a submanifold $X\subset Y$. Then $P=Q\cap\bC TX$, and smooth sections $\xi, \eta$ of $P$ can be extended, locally, to smooth sections $\txi, \teta$ of $Q$. Hence $[\xi, \eta]=[\txi, \teta] | X$ takes values both in $Q$ and in $\bC TX$, i.e. in $P$.

Second, suppose $X=Y\times Z$ and $\varphi$ is projection on $Y$. We check the above criterion for involutivity. Choose a collection $B$ of smooth 1--forms in a neighborhood $V$ of some $y\in Y$ such that
\[Q|V=\bigcap \{{\rm Ker} \beta : \beta\in B\} \quad{\rm{and}}\quad d\beta|Q_y=0\quad {\rm{for}}\quad \beta\in B.\]
Then the pullbacks $A=\{\varphi^*\beta :\beta\in B\}$ satisfy (3.1) and so $P$ is indeed involutive.

Finally, a general $\varphi$ can be factored
\[X\xrightarrow{{\rm{id}}_X \times\varphi} X\times Y \xrightarrow{{\rm{pr}}_{\rm{Y}}} Y.\]
With the subbundle $R\subset \bC T(X\times Y)$, 
\[R_{(x,y)}=\bC T_x X\oplus Q_y\subset \bC T_x X\oplus \bC T_y Y\simeq \bC T_{(x,y)} (X\times Y),\]
we can apply the two special cases to conclude $R$ and $P$ are involutive.
\end{proof}
  
 Suppose $C$ is a complex manifold and a smooth map
 \[G\times C\ni (g, c)\mapsto gc\in C\]
 defines a left action by holomorphic maps $C\to C$. Fix $c\in C$ and let
 \begin{equation}
 \varphi=\varphi_c : G\ni h\mapsto hc\in C.
 \end{equation}
Assuming that $\varphi_*^{-1} T^{10} C\subset \bC TG$ has constant rank, it defines an involutive subbundle $S=S(c)$ by Lemma 3.1. In fact $S(c)$ is left invariant. Indeed, suppose $h\in G$ and $\sigma\in S_h$. With $L_g:G\to G$ denoting  left translation by $g\in G$ we have $\varphi L_g=g\varphi$ and so $\varphi_*(L_{g*}\sigma)=(\varphi L_g)_*\sigma$ is the image of $\varphi_*\sigma\in T^{10}_{hc} C$ under the action of $g$. Since $G$ acts by holomorphic maps, $\varphi_*(L_{g*}\sigma)\in T^{10}_{ghc} C$ or $L_{g*} \sigma\in S_{gh}$ follows.

In the examples below $\varphi^{-1}_* T^{10}C$ will have constant rank for the following general reason. $G$ will be a regular closed subset and a submonoid of a Lie group $\Gamma$, and the action of $G$ on $C$ extends to a smooth action
\[\Gamma\times C\ni(g,c)\mapsto gc\in C\]
of $\Gamma$ by holomorphic maps. The map $\varphi_c(h)=\varphi(h)=hc$ is now defined for all $h\in\Gamma$, and the calculations we have just made show that $L_{g*}$ maps $\varphi_*^{-1}T^{10} C$ into itself for $g\in\Gamma$. Since the inverse of a left translation of $\Gamma$ is itself a left translation, left translations act as isomorphisms between the fibers of $\varphi^{-1}_*T^{10}C$, which therefore have the same dimension.

\noindent
\begin{example}
 Consider the group $\Gamma$ of affine transformations $t\mapsto a+bt$ $(a\in\bR, b\in (0,\infty))$ of $\bR$ and, given $r\in(0,\infty)$, the submonoid $G$ of those $g\in\Gamma$ that map $(-r,r)$ into itself. We take $C$ to be the complex plane, on which $\Gamma$ and $G$ still act by affine transformations extending the action on $\bR$. The involutive structures $S(c)$ for $c\in\bC\setminus \bR$ are complex structures and for $c\in\bR$ one dimensional foliations. 
If $a,b$ are used as coordinates to embed $G$ (or $\Gamma$) into $\bR^2$, the leaves of the foliation are line segments of slope $1/c$, while the complex structures are obtained by identifying $(a,b)\in \bR^2$ with $a+bc\in \bC$. For any $c\in\bC$, $S(c)$ is spanned by the vector $\oc\partial_a-\partial_b$. The $S(c)$ represent all left invariant involutive structures on $\Gamma$ of rank 1 but for one foliation, whose leaves are parallel with the $a$--axis. In \cite{LSz2} we have already worked with them in connection with geodesics of a Riemannian manifold.
\end{example}

\noindent
\begin{example}
Consider the Lie group $\rm{PSL}_2(\bC)$ acting by fractional linear transformations $g(\zeta)=(\alpha\xi +\beta)/ (\gamma\zeta+\delta)$ on the Riemann sphere $\bC\bP_1=\bC \cup\{\infty\}$. Let $G\subset \rm{PSL}_2(\bC)$ be regular closed and a submonoid. In Example 4.9 we construct such $G$ from intersecting bounded convex domains $D, E\subset\bC$; $G$ consists of those $g\in\rm{PSL}_2(\bC)$ that map both $D$ and $E$ into themselves. This $G$ is regular closed because $g$ that map $\oD$ into $D$ and $\oE$ into $E$ are in int$\,G$, and form a dense set in $G$. Choose integers $k, l\ge 0$ and let $G$ act diagonally  on $C=\bC\bP_1^k\times(\bC\bP_1^*)^l$ (here $\bC\bP_1^*$ is $\bC\bP_1$ with the opposite complex structure):
\[g(\zeta_1,\dots,\zeta_{k+l})=(g\zeta_1,\dots, g\zeta_{k+l}).\]
Depending on the choice of $k, l$ and $c\in C$ we obtain various left invariant involutive structures $S(c)$ on $G$ and in fact on PSL$_2(\bC)$. If $l=0$, for any $c$ the map $\varphi_c:\rm{PSL}_2(\bC)\to C$ is holomorphic when the source is given its standard complex structure, so that $S(c)\supset T^{10}\rm{PSL}_2(\bC)$. When among the coordinates $\zeta_1,\dots,\zeta_k\in\bC\bP_1$ of $c$ there are three different, $\varphi_c$ is a holomorphic embedding and $S(c)=T^{10}\psl_2(\bC)$.

In general, if among $\zeta_1\dots,\zeta_k, \overline{\zeta_{k+1}},\dots, \overline{\zeta_{k+l}}$ there are three different points and $c=(\zeta_1,\dots, \zeta_{k+l})$, then $\varphi=\varphi_c$ is a smooth embedding. This is  so because given three distinct points $\zeta, \omega,\tau\in \bC\bP_1$, associating with $g\in \psl_2(\bC)$ the triple $(g\zeta,g\omega, g\tau)$ sets up a (holomorphic) diffeomorphism between $\psl_2(\bC)$ and $(\bC\bP_1)^3$ minus the big diagonal. The image of $\varphi$ will be a CR manifold and $S(c)$ a CR structure, possibly trivial, $S(c)=(0)$.
However, if $\dim_{\bC}C=k+l<6=\dim_{\bR}\psl_2(\bC)$, the CR structure will not be trivial.
\end{example}

Let us return to a general manifold $N$ on which a Lie monoid $G$ acts on the right. We assume a left action of $G$ on a complex manifold $C$ as above. Further we assume that we are given another complex manifold $Z$ and a smooth map $\var:N\times C\to Z$ such that $\var^x=\var(x,\cdot)$ is holomorphic for all $x\in N$ and 
\begin{equation}  
\var(x,gc)=\var(xg,c),\quad x\in N, g\in G, c\in C.
\end{equation}  

\begin{thm} 
Let $c\in C$ and $\psi=\psi_c=\var(\cdot, c): N\to Z$. 
If $P=P(c)=\psi_*^{-1}T^{10}Z\subset \bC TN$ is a subbundle, then it is an involutive structure adapted to the action of $(G, S(c))$. 
\end{thm}  

\begin{proof}
By Lemma 3.1 $P$ is involutive, so all we need  to show is it is adapted. Suppose $x\in N$ and $\sigma\in S(c)$, i.e., 
$\varphi_*\sigma\in T^{10}C$. We rewrite (3.3) as $\var^x\circ\varphi=\psi\circ\Omega_x$, whence $\psi_*{\Omega_x}_*\sigma=
\var^x_*\varphi_*\sigma\in T^{10} Z$, since $\var^x$ is holomorphic. But this means ${\Omega_x}_*\sigma\in P$, and $P$ is indeed adapted.
\end{proof}

\section{Examples}

We illustrate Theorem 3.4 on three examples. They all involve real analytic objects. In what follows we will drop the qualifier `real' of `analytic'.

\noindent
\begin{example}
Consider an analytic manifold $M$ with an analytic connection $\nabla$ on $TM$. There is a complex manifold $M^\bC$ of which $M$ is a maximally real,
analytic submanifold, in the sense that locally in source and target, the pair $(M, M^\bC)$ is biholomorphic to a real Banach space and its 
complexification. When $\dim M<\infty$ this is due to Bruhat and Whitney \cite{WB}, whose construction Patyi and Simon \cite{PS} extended to Banach 
manifolds. Upon shrinking $M^\bC$ we can arrange that $\nabla$ extends to a holomorphic connection $\nabla^\bC$ on $M^\bC$.

A geodesic of $\nabla$ is a smooth map $x:[p,q]\to M$ of some interval $[p,q]\subset\bR$ whose velocity vector $\dot x$ is parallel. In local charts
geodesics satisfy a second order analytic ODE and this implies that they extend to a holomorphic map of a neighborhood of $[p,q]\subset\bC$ into $M^\bC$

Fix an $r\in (0,\infty)$, and as in Example 3.2, let $G$ consist of strictly increasing affine transformations of $\bR$ that leave $(-r,r)$ 
invariant.
This monoid acts on $\bC$ as well; choose a simply connected bounded open neighborhood $C$ of $0\in\bC$  that $G$ leaves invariant. For example, 
$C$ could be a disc $\{\zeta\in\bC: |\zeta|<R\}$ of radius $R\ge r$. Let $N$ consist of those geodesics $x:[-r,r]\to M$ that extend to a holomorphic 
map $\tx$ of a neighborhood of $\oC$ into $M^\bC$. Associating with $x\in N$ its velicity $\dot x(0)\in TM$ embeds $N$ as an open subset into $TM$ 
and endows $N$ with the structure of an analytic Banach manifold. For example, all zero speed geodesics are in $N$. The formula
\[N\times G\ni (x,g)\mapsto xg=x\circ g\in N, \]
defines the action of $G$ on $N$. Let $Z=M^\bC$, 
\[\var:N\times C\ni(x,c)\mapsto\tx(c)\in M^\bC,\qquad{\rm{and}}\qquad \psi=\psi_c=\var(\cdot, c).\]
We are in the situation of Theorem 3.4: provided $P=\psi^{-1}_*T^{10}Z$ is a Banach bundle, it will be an involutive structure on $N$ adapted to the action of $(G, S(c))$, $c\in C$. In fact, $P$ obtained in this example is always a Banach bundle, as we presently verify in the following, more general example.
\end{example}

\noindent
\begin{example}
While Example 4.1 was about trajectories of certain second order ODEs, here we will consider trajectories of rather general first order ODEs 
(vector fields). In this generality there is not much action on the space of trajectories, and we will work with $G=\{e\}$. Although most
results of this paper become vacuous when $G$ is trivial, Theorem 3.4 still has content, at least the part about involutivity. The point is
that in this example, without any meaningful action, it is still possible to show that $P(c)\subset\bC TN$ is a subbundle.

The set up is as follows. We consider an analytic direct submersion of analytic manifolds $\pi: X\to M$. Our manifold $N$ will consist of certain 
trajectories of a fixed real analytic vector field $\xi$ on $X$. 

A subclass of examples to keep in mind is when $\pi$ is the bundle projection $X=TM\to M$ and 
\begin{equation} 
\pi_*\xi(x)=x\quad{\rm{for }} \quad x\in TM.
\end{equation}   
If $M$ is an open subset of a Banach space $E$, then $TM=M\times E$, and (4.1) means
\begin{equation}   
\xi(u,v)=(v,F(u,v)), \quad u\in M, v\in E,
\end{equation}   
with some analytic $F:TM\to E$. In (4.2) we viewed the vector field $\xi$ as a map $TM\to E\times E$, without indicating the footpoint $(u,v)$ of the vector $\xi(u,v)$. It follows that trajectories $x:[p,q]\to TM$ of $\xi$ project to $M$ as solutions of the second order ODE
\begin{equation}   
\ddot u=F(u,\dot u).
\end{equation}    
Even with a general manifold $M$, trajectories of $\xi$ on $TM$ are in one-to-one correspondence with trajectories of an $M$ valued second order ODE. In this way Example 4.2 generalizes Example 4.1.

As before we can complexify our set up and find a direct holomorphic submersion $\pi^\bC:X^\bC\to M^\bC$ of complex manifolds and a holomorphic vector field $\xi^\bC$ on $X^\bC$ (i.e., a section of $T^{10}X$) so that $X, M$ are maximally real, analytic submanifolds of $X^\bC, M^\bC$, $\pi^\bC|X=\pi$, and $2{\rm{Re}}\, \xi^\bC|X=\xi$.

Fix next a connected, simply connected, bounded open neighborhood $C$ of $0\in \bC$, and let $N$ consist of trajectories $x:(-\var,\var)\to X$ of $\xi$ ($\var>0$ is not fixed) that extend to holomorphic maps $\tx$ of a neighborhood of $\bar C$ into $X^\bC$. Of course, the extension $\tx$ is a trajectory of $\xi^\bC$. Associating with $x\in N$ its initial value $x(0)\in X$ realizes $N$ as an open subset of $X$, and endows $N$ with the structure of a Banach manifold. It is possible that $N$ turns out to be empty. However, given any trajectory $x:(-\var,\var)\to X$ of $\xi$, if $C\subset \bC$ is chosen a sufficiently small neighborhood of $0$, $N$ will contain  $x$.
\end{example}

As in Example 4.1 we fix $c\in C$ and with $x\in N$ let $\psi(x)=\psi_c(x)=\pi^\bC(\tx(c))\in M^\bC$.

\begin{prop} 
$P=P(c)=\psi^{-1}_*T^{10}M^\bC\subset \bC TN$ is a complemented subbundle.
\end{prop}   

\begin{proof} 
 Let $N^\bC$ stand for all holomorphic trajectories $y$ of $\xi^\bC$, defined in some neighborhood of $\oC$. Thus 
 $N=\{y\in N^\bC: y(0)\in X\}$. As with $N$, the map $y\mapsto y(0)$ identifies $N^\bC$ with an open subset of $X^\bC$, and endows it with the
structure of a complex manifold in which $N$ is a maximally real, analytic submanifold. The fundamental theorem of ODEs implies that the map
\[H:N^\bC\ni y\mapsto y(c)\in X^\bC\]
is a biholomorphism on its image, an  open subset of $X^\bC$. It follows that
\[\psi^\bC=\pi^\bC\circ H: N^\bC\to M^\bC\]
is a direct holomorphic submersion and $\psi^\bC | N=\psi$. Let $J_M,J_N$ denote the complex structure tensors on $M^\bC,N^\bC$. 
There is an $\bR$--isomorphism $l:\bC TN\to TN^\bC|N$ given by
 \[l(v)=v,\qquad l(iv)=-J_Nv,\qquad\text{for}\quad v\in TN\subset\bC TN.\] 
As $\psi^\bC$ is holomorphic, for $v,w\in T_yN$
$$
\psi_*^\bC(lv)=\psi_*v,\qquad -\psi_*^\bC l(iw)=\psi_*^\bC(J_Nw)=J_M\psi_*^\bC w=J_M\psi_*w.
$$
Therefore $\psi_*(v+iw)\in T^{10}M^\bC$, or $J_M\psi_*w=\psi_*v$, if and only if $\psi_*^\bC l(v+iw)=0$. Since $\psi^\bC$ is a direct 
submersion,
\[P=\psi_*^{-1}T^{10}M^\bC=l^{-1}\Ker \,\psi_*^\bC|(TN^\bC|N)\]
is indeed a complemented subbundle of $l^{-1}(TN^\bC|N)=\bC TN$.

\end{proof}

Thus a combination of Theorem 3.4 and Proposition 4.3 constructs involutive structures $P(c)$ on spaces $N$ of trajectories of the vector field $\xi$. The nature of $P(c)$ depends on $c$ and on the interaction between $\xi$ and $\pi$ in a complicated way. In general we only have a perturbative result:

\begin{thm} 
Suppose $\pi_*\circ\xi: X\to TM$ is a local diffeomorphism and $x\in N$. If $c\in C\setminus\bR$ is sufficiently close to $0$, then $P(c)$ is 
a complex structure: $P(c)\oplus\overline{P(c)}=\bC TN$, over a neighborhood of $x$.
\end{thm} 

The assumption concerning $\pi_*\circ\xi$ is met if $X=TM$ and $\xi$ comes from a second order autonomous ODE valued in $M$, see (4.1), (4.2). 
Conversely, if $\pi_*\circ\xi$ is a global diffeomorphism $X\to TM$, then the pushforward of $\xi$ by this differomorphism is a vector field 
$\xi^1$ on $TM$ that satisfies (4.1), and so its trajectories are lifts of solutions of a second order ODE of type (4.3).

\begin{proof}
A calculation in charts will show that for $c\in C\setminus\bR$ close to 0 the map $\psi=\psi_c$ is a diffeomorphism of a neighborhood of 
$x\in N$ on an open subset of $M^\bC$. We can assume that a neighborhood of $x(0)\in X$ is an open subset $U$ of a Banach space $E$, a 
neighborhood of $\pi(x(0))\in M$ is an open subset $V$ of a Banach space $F$, and $\pi|U$ is the restriction of a real linear map $p:E\to F$. 
This means that in $X^\bC$, $M^\bC$, neighborhoods of $x(0)$, $\pi(x(0))$ can be taken open subsets of $\bC\otimes E$, $\bC\otimes F$. Over 
$U$ and $V$ the tangent bundles $TX, TM$ are $U\times E$, $V\times F$, and in this representation
\[\text{if } \xi(u)=(u,\alpha(u))\in U\times E,  \text{ then } (\pi_*\circ\xi)(u)=(\pi u, p\alpha(u))\in V\times F.\]
Since $\pi_*\circ \xi$ is a local diffeomorphism, the restrictions of $p\alpha: U\to F$ to the fibers of $\pi$ are local diffeomorphisms.

By Taylor's formula, for $y\in N$ in a neighborhood of $x$
\[\ty(s)=y(\text{Re}\,s)+i(\text{Im}\,s)\alpha(y(\text{Re}\,s))+O(\text{Im}\,s)^2, \text{ and }\]
\begin{equation}  
\pi^\bC(\ty(s))=\pi(y(\text{Re}\,s))+i(\text{Im}\,s)p\alpha(y(\text{Re}\,s))+O(\text{Im}\,s)^2,
\end{equation} 
as $\bC\ni s\to 0$. The error term is uniform for $y$ close to $x$ and $s$ close to $0$; in fact, viewing both sides as functions of $y\in N$, 
the error in $C^1$ (or any $C^k$) norm is of order $(\text{Im}\,s)^2$. When $\text{Im}\,s\ne 0$, $\psi$ is a diffeomorphism near $x$ if its 
composition with
\[L:\bC\otimes F\ni v\mapsto \text{Re}\,v+{\rm{Im}}\, v/{\rm{Im}} s\in \bC\otimes F\]
is. By (4.4)
\[L\pi^\bC(\ty(s))=\pi(y(\text{Re}\,s))+ip\alpha(y(\text{Re}\,s)) +O(\text{Im}\,s).\]
When $\text{Re}\,s=0$, the first term on the right, $y\mapsto \pi(y(0))$, defines a direct submersion $N\to F$ near $x$ and, as said, the second term, restricted to the fibers of this submersion, is a local diffeomorphism to $iF$. Hence the first two terms represent near $x$ a local diffeomorphism $N\to \bC\otimes F$. It follows that when $c\in C\setminus \bR$ is sufficiently close to 0, $y\mapsto L\pi^\bC(\ty(c))=L\psi(y)$ is a diffeomorphism near $x$, and so is $\psi$.

Therefore the pullback of $T^{10}M^\bC$ by $\psi$ is a complex structure on $N$ near $x$ (and with this structure $N$ is locally biholomorphic to open sets in complex Banach spaces).

\end{proof}

The construction of adapted complex structures in \cite{Bi,HK2, Sz1, Sz4} can be obtained from Example 4.2. Those structures correspond to the 
choice $c=i$, whereas the involutive structures above are complex only if $c\in \bC\setminus \bR$ is sufficiently small. This difference can be 
overcome by a certain scaling. The second order $M$ valued ODEs in \cite{Bi, HK2, Sz1,Sz4} have constant maps as solutions. By rescaling, using 
Example 4.2 one obtains a result of the following nature: Any $x\in N$ representing a constant map $[-r,r]\to M$ has a neighborhood on which 
the involutive structure $P(i)$ is a complex structure.

The adapted complex structures of a Riemannian manifold $M$, and the magnetic complex structure of Hall and Kirwin \cite{HK2} when $M$ is 
endowed with a closed 2-form, are even K\"ahler. This generalizes to the framework of Example 4.2.

\begin{thm} 
In the setting of Theorem 4.4, view $N$ as an open subset of $X$. Suppose $\omega$ is a 2-form on $X$, invariant under the local flow of 
$\xi$. If $\omega$ vanishes on the fibers of $\pi$, then $\omega$ is of type $(1,1)$ with respect to the complex structure $P(c)$.
\end{thm}  

\begin{proof}
We need to show that $\omega|P(c)=0$ and $\overline{\omega}|P(c)=0$; but the latter will follow once the former is proved. This we will do in 
the general setting of 
Example 4.2 (so $\pi_*\circ \xi$ is not necessarily a local diffeomorphism and $P(c)$ need not be a complex structure). First assume 
$c\in C\cap \bR$. Then $\psi_c(y)=\pi(y(c))\in M$ for $y\in N$ and
\[\psi_{c*}\bC TN\subset \bC TM\subset\bC TM^\bC | M.\]
Since $M\subset M^\bC$ is maximally real, $\bC TM$ and $T^{10}M^\bC | M$ intersect each other in the zero section, and so
\[P(c)=(\psi_{c*})^{-1} T^{10}M^\bC=\text{Ker}\,\psi_{c*}.\]
In other words, $P(c)$ is the complexified tangent bundle to the fibers of $y\mapsto \pi(y(c))$. By our assumption $\omega | P(c)=0$ when 
$c=0$; and since for $c\in\bR$ small $(c,y)\mapsto y(c)$ is the flow of $\xi$, invariance of $\omega$ implies the same for such $c$. Let $c$ 
now vary in $C$. For fixed $y\in N$, as a point in the Grassmannian of complemented subspaces of $\bC T_yN$, $P(c)_y$ varies holomorphically. 
Therefore by analytic continuation $\omega | P(c)_y=0$ for $c\in C$.
\end{proof}

\noindent
\begin{example}
Let $\oD, M$ be finite dimensional smooth Riemannian manifolds, $\oD$ compact with boundary, $M$ without boundary. A $C^2$ map $x:\oD\to M$ is 
harmonic if it is a critical point of the Dirichlet energy functional for maps with fixed boundary values on $\partial D$. This means that its 
tension field $\tau(x)$ is zero; here $\tau$ is a certain second order partial differential operator. When $\dim D=2$, both Dirichlet energy and
$\tau$ stay the same if the metric of $\oD$ is multiplied by a smooth function, hence one can talk about harmonicity of a map from a Riemann
surface into a Riemannian manifold; and the composition of a harmonic map with a conformal map $\oD\to \oD$ is also harmonic. In this example
the manifold $N$ will consist of certain harmonic maps between analytic Riemannian manifolds, but to ensure the maps considered form a
Banach manifold,  and when the source is two dimensional, conformal maps act continuously, we have to proceed with some care.

Fix a number $p>2$ that is not an integer. H\"older maps $\oD\to M$ of class $C^p$ form a Banach manifold $C^p(\oD, M)$. However, we will work
with the less common manifold of little H\"older maps $c^p(\oD, M)$ instead, which can be defined as the closure of $C^\infty(\oD,M)$ in
$C^p(\oD, M)$. Equivalently, a map $x:\oD\to M$ is of class $c^p$ if it is $[p]$ times differentiable and its partial derivatives $y$ of order
$[p]$ satisfy
\[\lim\limits_{s\to t}\frac{|y(s)-y(t)|}{|s-t|^{\{p\}}}=0, \quad \{p\}=p-[p].\]
Here partial derivatives and $|s-t|$ are computed in local coordinates. As with $C^p(\oD, M)$, the tangent space of $c^p(\oD, M)$ at some 
$x$ can be identified with $c^p(\oD, x^* TM)$, the space of $c^p$ sections of the pullback bundle $x^*TM$. 

If $\lam \mapsto x_\lam$ is a smooth curve in $c^p(\oD, M)$ through a harmonic $x=x_0$, whose velocity $dx_\lam / d\lam$ at $\lam=0$ is a tangent vector $\xi\in T_xc^p(\oD, M)$, one can compute the directional derivative of the tension $\tau'(\xi)=d\tau(x_\lam)/d\lam$ at $\lam=0$. In the natural identification $T_x c^p(\oD,M)\approx c^p(\oD, x^*TM)$ it turns out that 
\[\tau'=\tau'_x: c^p(\oD, x^* TM)\to c^{p-2}(\oD, x^*TM)\]
is a second order linear partial differential operator, a Laplace operator on $x^*TM$. The restriction of $\tau'_x$ to
\[\{\xi\in c^p(\oD, x^* TM): \xi | \partial D=0\}\]
is Fredholm of index $0$, see \cite{Go}. (Goldstein does the computations for $C^p$ maps but they work the same for $c^p$ maps.) If this 
restriction is an isomorphism, we will call the harmonic map $x$ persistent. Clearly, harmonic maps close to persistent ones are themselves
persistent. Goldstein shows that if $M$ has nonpositive sectional curvatures, then every harmonic map is persistent. His argument gives that
regardless of curvature, constant maps are always persistent. His subsequent analysis (an application of the implicit function theorem) gives
that in $c^p(\oD, M)$ persistent harmonic maps form a not necessarily closed submanifold that we denote $N^0$; and the map
\[N^0\ni x\mapsto x | \partial D\in c^p(\partial D, M)\]
is a local diffeomorphism. Further, $T_xN^0$ consists of $\xi\in T_x c^p(\oD, M)$ that $\tau'$ annihilates.
\end{example}

We now assume that $\oD, M$ are analytic Riemannian manifolds. According to a general theorem of Morrey, harmonic maps are analytic on $D$.

\begin{thm} 
Any $x\in N^0$ has a neighborhood $N$ and there are complexifications $D^\bC\supset D$, $M^\bC\supset M$ such that every $y\in N$ extends to a holomorphic map $\ty: D^\bC\to M^\bC$. If $k\in \bN$, $D^\bC$ can be chosen so that the map $\psi: N\to (M^\bC)^k$ defined by
\[\psi(y)=\psi_c(y)=(\ty(c_1),\dots, \ty(c_k)), \quad c=(c_1,\dots, c_k)\in (D^\bC)^k\]
is smooth. Finally, if $c_j^0\in D$ are distinct for $j=1,\dots, k$, and $c\in (D^\bC)^k$ is sufficiently close to $c^0=(c^0_1,\dots, c^0_k)$, then $P(c)=P=\psi_*^{-1} T^{10}(M^\bC)^k$ is an involutive structure of co-rank  $=k\dim M$ in a neighborhood of $x$ in $N$. The neighborhood depends on $c^0$ but not on $c$ sufficiently close to it.
\end{thm} 

\begin{proof}
The first statement is what Morrey proves for $C^p$ solutions of elliptic systems of analytic PDEs, of which the harmonic map equation 
$\tau(y)=0$, in local coordinates, is an instance. The construction of the holomorphic extension $\ty$ of $y$, by a certain iteration, is 
such that $\ty\in C^p(\overline{D^\bC}, M^\bC)$ depends smoothly on $y$. Hence $\psi$ is smooth.

It remains to show that under the assumptions $P\subset \bC TN$ is a subbundle, for involutivity then follows from Lemma 3.1. We claim that 
$\psi_c$ is a direct submersion for $c=c^0$, at least in a neighborhood of $x$; in other words, $\psi_{c*} | T_x N$ is surjective and its 
kernel is complemented. The latter is clear, 
the kernel is finite codimensional. For surjectivity we need to produce $\xi\in T_xN\approx c^p(\oD, x^*TM)$, annihilated by $\tau'$, that 
takes given values at $c^0_j$, or at least values near given ones. The Cauchy--Kovalevskaya theorem produces such a section, if not on all of 
$\oD$ but on some neighborhood of 
each $c^0_j$. The existence of a global $\xi$ follows from Malgrange's approximation theorem \cite[p.341]{Ma}.

Once $\psi_c$ is known to be a direct submersion near $x$ when $c=c^0$, the same follows for $c$ in a neighborhood of 
$c^0$, and the rest  of the theorem follows from  Lemma 3.1. 
\end{proof}

\noindent
\begin{example}
In Example 4.6 there was no action. But suppose $D\subset \bC$ is a bounded convex domain. If $x\in c^p(\oD, M)$ is 
harmonic and $g:\oD\to\oD$ is a smooth conformal map, the conformal invariance of Dirichlet energy implies
that  $x\circ g$ is also harmonic. If $x$ is close to constant, then so is $x\circ g$, hence persistent.  Thus on a neighborhood
$N^1\subset N^0$ of constant maps, reparametrizations define an action of the Lie monoid $G^0\subset \psl_2(\bC)$ of those automorphisms $g$ of $\bC\bP_1$ that map $D$ into itself:
\begin{equation} 
N^1\times G^0\ni (x,g)\mapsto xg=x\circ g\in N^1.
\end{equation} 
One easily checks that this is an action in the sense of Definition 2.2. (It is to ensure the continuity of (4.5) that we 
chose to work with 
little H\"older maps.) We will now show that the construction in 
Example 4.6 gives rise to involutive structures on neighborhoods of any $x\in N^1$ that are adapted to left invariant structures $S(c)$ on 
certain submonoids $G\subset G^0$ discussed in Example 3.3.

Fix $k\in\bN$ and, given $x\in N^1$, choose a complexification $D^\bC\supset D$ and a neighborhood $N$ of $x$ as in Theorem 4.7. If $\bC^*$ 
denotes $\bC$ with its opposite complex structure, we can realize this complexification as an open subset $D^\bC \subset\bC\times \bC^*$ into 
which $D$ is embedded diagonally. Fix also a bounded convex neighborhood of $(0,0)\in D^\bC$ of the form $E\times E$. The group $\psl_2(\bC)$ 
acts diagonally on $\bC\bP_1\times \bC\bP^*_1$; let $G$ be the Lie monoid of those $g\in\psl_2(\bC)$ that map $D$ and $E\times E$ into 
themselves. If $c_j=(\zeta_j,\zeta_{j+k})\in \bC\times \bC^*$ and $c=(\zeta_1,\dots, \zeta_{2k})$, on the one hand Example 3.3 defines an involutive 
structure $S(c)$ on $G$. On the other hand, for certain $c$ Theorem 4.7 defines an involutive structure $P(c)$ on a 
neighborhood of $x$. Theorem 3.4, with a suitable $C\subset (D^\bC)^k$ and 
\[\var: N\times C\ni (y,c)\mapsto (\ty(c_1),\dots, \ty(c_k))\in (M^\bC)^k\]
then gives that $P(c)$ is adapted to the action of $(G,S(c))$. 
\end{example}

\section{Geometry in involutive manifolds}

The second half of this paper addresses uniqueness of adapted involutive structures. One ingredient in the proof of uniqueness involves
geometric notions in involutive manifolds, and this will be the topic of this section and the next. In the special case of complex manifolds
the material will be familiar. The main result of these two sections is Theorem 6.2, a characterization of variations $df_\lam/d\lam$ in
families of involutive maps $f_\lam:(M,Q)\to (N, P)$. The characterization is in terms of a partial connection, a notion we now introduce. We will discuss partial connections on complex Banach bundles $\pi:E\to N$ over involutive manifolds 
(although the proper generality would be just direct submersions $\pi$ with base and fibers given involutive structures). 
The 
vertical tangent bundles $\bC T^{\rm{vert}}E, \bC T^{\rm{vert}10}E, \bC T^{\rm{vert}01}E\subset \bC TE$ consist of complexified 
tangent vectors to the fibers and $(1,0)$, respectively $(0,1)$, vectors tangent to the fibers.

To begin, for any complex Banach space $Z$ and $z\in Z$, $\zeta\in\bC T_zZ$ we define $\iota_z(\zeta)=\zeta\,\id_Z\in Z$ (the $\zeta$-derivative of $\id_Z$). For example  $\iota_z|T_z^{01}Z=0$ while $\iota_z|T_z^{10}Z$ is an isomorphism.
\begin{lem}
If $\sigma:N\to Z$ is a $C^1$ map, $x\in N$, and $v\in\bC T_xN$, then $\iota_{\sigma(x)}(\sigma_*v)=v\sigma$.
\end{lem}
\begin{proof}
Indeed, $(\sigma_*v)\id_Z=v(\id_Z\circ\sigma)=v\sigma$.
\end{proof}

Applying this to the fibers of a complex Banach bundle $E$ we obtain a smooth map 
\begin{equation}
\iota=\iota_E:\bC T^{\rm{vert}}E\to E,
\end{equation}
that restricts to linear maps $\bC T_e^{\rm{vert}}E\to E_{\pi(e)}$ for $e\in E$.

\begin{defn}
A partial or $\oP$--connection on $\pi:E\to (N,P)$ is a smooth map $\Pi:\pi_*^{-1}\oP\to E$ such that $\Pi|(\pi_*^{-1}\oP)_e$ is a linear map to $E_{\pi(e)}$ for all $e\in E$, and
\begin{equation*}
\Pi|\bC T^{\rm{vert}}E=\iota.
\end{equation*}
\end{defn}

As in more familiar situations, a partial connection allows one to differentiate sections. Suppose  $s$ is a $C^1$ section 
of $E$ in a 
neighborhood of $x\in N$. We let
\begin{equation}  
\nabla_vs=\Pi s_*v\in E_x,\qquad v\in\oP_x.
\end{equation} 
This makes sense since $\pi_*s_*v=v$ and so $s_*v\in \pi_*^{-1}\oP$. Clearly, the differential operator $\nabla$ and $\Pi$ determine each other. We will also call $\nabla$ a partial connection.

Suppose now that $E=N\times Z\to N$ is a trivial bundle. A $\oP$--connection on $E$ determines a smooth map $\theta:\oP\times Z\to Z$ as follows. Given $v\in\oP_x$ and $z\in Z$, we let $0_z\in \bC T_zZ$ be the zero vector, and define $\theta$ by
\begin{equation}
\big(x,\theta(v,z)\big)=\Pi(v,0_z)\in E_x=\{x\}\times Z.
\end{equation}
Here and below we identify tangent spaces to products with products of the appropriate tangent spaces, so  
$(v,0_z)\in\bC T_xN\oplus\bC T_zZ\approx\bC T_{(x,z)}(N\times  Z)$.
Hence if $\zeta\in\bC T_zZ$,
$$
\Pi(v,\zeta)=\Pi(0,\zeta)+\Pi(v,0_z)=\iota_z(\zeta)+\big(x,\theta(v,z)\big).
$$
If $s(y)=\big(y,\sigma(y)\big)$ defines a $C^1$ section in a neighborhood of $x\in N$, then
\begin{gather}
\Pi s_*v=\Pi(v,\sigma_*v)=\big(x,\iota_{\sigma(x)}(\sigma_*v)+\theta(v,\sigma(x))\big),\qquad\text{or}\notag\\
\nabla_v s=\big(x, v\sigma+\theta(v,\sigma(x))\big)\in\{x\}\times Z=E_x 
\end{gather}
by Lemma 5.1. Clearly $\theta(v,z)$ and $\nabla_vs$ are linear in $v\in\oP_x$.

\begin{defn}
A partial connection $\Pi$ on a Banach bundle $E\to (N,P)$ is linear if $\nabla_vs$ is $\bC$--linear in $s$ as well. Equivalently, in any local trivialization $E|U\approx U\times Z$ ($U\subset N$ open), $\theta(v,z)$ is $\bC$--linear in $z$ as well.
\end{defn}   

If $\Pi$ is linear, in any trivialization $\theta$ can be viewed as a smooth map $\oP\to \text{End}\, Z$ (the Banach 
space of operators in $Z$), which is linear on each $\oP_x$. In other words, it is an $\text{End}\,Z$ valued relative 1-form, 
the connection form of the partial connection.

In finite dimensions linear partial connections were introduced by Bott when $P=\oP$ and by Rawnsley for certain other 
involutive structures, both through the associated covariant differentiation $\nabla$, see \cite{Bo, R}.

A partial connection $P$ on $E\to(N,P)$ can be pulled back along an involutive map $g:(M,Q)\to (N,P)$. Indeed, 
$\varrho : g^*E\to M$ comes with a canonical smooth map $\hat g: g^*E\to E$ covering  $g$, that is an isomorphism between 
fibers. The pull back connection $g^*\Pi: \rho_*^{-1}(\oQ)\to g^*E$ is the connection defined by
$$
\hat g(g^*\Pi)(\xi)=\Pi(\hat g_*\xi)\qquad\text{for}\quad \xi\in\rho_*^{-1}\oQ.
$$
When $E=N\times Z\to N$ is trivial and $\theta$ is as in (5.3), $\theta^g:\oQ\times Z\to Z$ corresponding to the $\oQ$--connection 
$g^*\Pi$ on $g^*E=M\times Z\to M$ is $\theta^g(u,z)=\theta(g_*u,z)$. Passing to local trivializations this implies for general 
bundles $E\to N$ and the pull back connection $\nabla^g$ that for sections $s$ of $E$ over a neighborhood of $g(x)$
\begin{equation}  
\nabla_u^g(g^*s)=g^*\nabla_{g_*u}s,\qquad u\in\oQ_x.
\end{equation} 
In particular, the pull back of a linear connection is itself linear.

Consider next along with the involutive manifold $(N,P)$ the Banach bundles 
\[\bC TN\to N,\qquad \pi:P'=\bC TN / \oP \to N.\]
We will denote the projection of $u\in \bC TN$ on $P'$ by $u'$, and likewise for sections of $\bC TN$. The bundle $P'$ has a 
natural $\oP$--connection, obtained via the following lemma.

\begin{lem}
Let $\xi\in\pi_*^{-1}\oP\subset\bC TP'$ and $\pi_*\xi=v\in\oP_x\setminus(0)$, $x\in N$. If $s$ and $w$ are smooth sections 
of $\bC TN$, respectively $\oP$, in a neighborhood of $x$ such that $s'_*v=\xi$ and $w(x)=v$, then
\begin{equation}
\Pi(\xi)=[w,s]'(x)\in P'_x
\end{equation}
depends only on $\xi$, and not on the choice of $s,w$.
\end{lem}

\begin{proof}
If $\alpha$ is a smooth 1-form on $N$ and $w,s$ are smooth sections of $\bC TN$, 
\begin{equation}  
d\alpha(w, s)=w\alpha(s)-s\alpha(w)-\alpha[w,s].
\end{equation} 
Suppose $\alpha$ vanishes on $\oP$ and $w$ is a section of $\oP$. If $w$ vanishes at some $x\in N$, then the first three 
terms in (5.7) vanish at 
$x$; hence so does $\alpha[w,s]$. As at the price of shrinking $N$ we can arrange that $\oP$ is the intersection of Ker $\alpha$ for a family of
1-forms $\alpha$, it follows that $[w,s](x)\in \oP_x$, and $[w,s]'(x)=0$. This implies that as far as dependence on $w$, 
$[w,s]'(x)$ depends only on $w(x)=\pi_*\xi$.

At the same time $[w,s]'(x)$ depends on $s$ only through $s'_*v$. Indeed, let $t$ be another section of
$\bC TN$ near $x$ with $t'_*v=\xi$.  This implies $s'(x)=t'(x)$. To show $[w,s]'(x)=[w,t]'(x)$ we can even assume $s(x)=t(x)$, for if we 
add to $t$ a section of $\oP$, then by involutivity $[w,t]'$ does not change. Let
$\beta\in\Gamma({P'}^*)$ (section of the dual of $P'$); then $\alpha(u)=\beta(u')$ 
defines a $1$--form $\alpha$ on $N$ that vanishes on $\oP$. From (5.7)
$$
\beta[w,s]'(x)=v\beta(s')-s(x)\alpha(w)-d\alpha\big(v,s(x)\big)
$$
and similarly for $t$. Subtracting the two yields $\beta[w,t]'(x)-\beta[w,s]'(x)=v\beta(t')-v\beta(s')=0$ by the chain rule. Since this is true for any $\beta$, $[w,s]'(x)=[w,t]'(x)$ follows.
\end{proof}
\begin{lem}
$\Pi:\pi_*^{-1}\oP\setminus\bC T^{\rm{vert}}\to P'$ defined in Lemma 5.4 extends to a linear connection $\pi_*^{-1}\oP\to P'$. 
The corresponding covariant differentiation is given by
\begin{equation}
\nabla_ws'=[w,s]'.
\end{equation}
\end{lem}
The formula generalizes Bott's construction of a partial connection on the normal bundle of a foliation of a finite dimensional 
manifold. When $(N,P)$ is a complex manifold, $P'$ can be identified with $P=T^{10} N$ and $\nabla$ becomes the Dolbeault 
operator on the holomorphic vector bundle $P'\approx T^{10}N\to N$. We will call $\Pi$ or $\nabla$ the Bott connection on 
$P'\to N$; it is linear by (5.8).
\begin{proof}
Let us write $\iota=\iota_{P'}$ and $\tilde\iota=\iota_{\bC TN}$, cf. (5.1). If we define the extension to 
$\bC T^{\rm{vert}}P'=\Ker\,\pi_*\to P'$ by $\Pi|\bC T^{\rm{vert}}P'=\iota$, the extended
$\Pi$ satisfies
\begin{equation}
\Pi(\xi+\eta)=\Pi(\xi)+\iota(\eta),\qquad\text{if}\quad \xi\in(\pi_*^{-1}\oP)_p,\, \eta\in\bC T_p^{\rm{vert}}P',\, p\in P'.
\end{equation}

Indeed, let $\pi_*\xi=v\in\oP_x$. If $v=0$, (5.9) is just linearity of $\iota$; if $v\in\oP_x\setminus(0)$, we compute the 
quantities in (5.9) as follows.
We assume $N$ is an open subset of a real Banach space $Z$ so that $\bC TN$ is $N\times(\bC\otimes Z)$.
We lift $\xi,\eta$ to
$\tilde\xi\in\bC T_{\tilde p}(\bC TN)$ and $\tilde\eta\in\bC T^{\rm{vert}}_{\tilde p}(\bC TN)$ with some $\tilde p\in \bC TN$, 
and choose 
$w\in\Gamma(\oP)$, $s,t\in\Gamma(\bC TN)$ so that $w(x)=v$ and $s_*v=\tilde\xi$, $t_*v=\tilde\xi+\tilde\eta$. 
In particular, $s(x)=t(x)$.
Let $w(y)=\big(y,\omega(y)\big)$, $s(y)=\big(y,\sigma(y)\big)$, $t(y)=\big(y,\tau(y)\big)$ with smooth functions 
$\omega,\sigma,\tau:N\to \bC\otimes Z$. Then 
$$
\tilde\iota(\tilde\eta)=\tilde\iota(t_*v-s_*v)=\tilde\iota\big((v,\tau_*v)-(v,\sigma_*v)\big)=\tilde\iota(0,\tau_*v-\sigma_*v)=\big(x,v(\tau-\sigma)\big)
\in\bC T_xN
$$
by Lemma 5.1. Furthermore $[w,s](x)=\big(x,v\sigma-\sigma(x)w\big)$ and similarly for $[w,t]$, whence
$$
\Pi(\xi+\eta)-\Pi(\xi)=[w,t-s]'(x)=\big(x,v(\tau-\sigma)\big)'.
$$
Since $\tilde\iota(\tilde\eta)'=\iota(\eta)$, (5.9) follows.

To prove that $\Pi$ is smooth, and linear on $(\pi_*^{-1}\oP)_p$, $p\in P'$, at the price of shrinking and choosing a 
different  trivialization we can assume that $\bC TN=N\times A$ and $\oP=N\times B$, with $A\supset B$ a Banach space 
and its subspace. Thus $P'=N\times Z$, with $Z=A/B$. If $x\in N$, $v\in\oP_x$, $a\in A$, and 
$z=a'\in Z$, let $s,w$ be constant sections of $N\times A\to N$ and 
$N\times B\to N$ such that $s(x)=(x,a)$ and $w(x)=v$, then let
\begin{equation}
\big(x,\tilde\theta(v,a)\big)=[w,s](x) \in\bC T_xN,\qquad \big(x,\theta(v,z)\big)=[w,s]'(x)\in P'_x.
\end{equation}
By Lemma 5.4 $[w,s]'(x)$ indeed depends only on $v$ and $z$. Clearly $\tilde\theta:\oP\times A\to\bC TN$ is smooth, 
and by Lemma 5.6 below so is $\theta:\oP\times Z\to P'$. Next, if 
$\xi=(v,\zeta)\in(\pi_*^{-1}\oP)_{(x,z)}\subset\bC T_{(x,z)}P'\approx\bC T_xN\oplus\bC T_zZ$ and 
$0_z\in\bC T_z Z$ is the zero vector, by (5.9)
$$
\Pi(v,\zeta)=\iota(0,\zeta)+\Pi(v,0_z)=\big(x,\iota_z(\zeta)+\theta(v,z)\big)\in P'_x.
$$
Thus $\Pi|(\pi_*^{-1}\oP)_{(x,z)}$ is linear and $\Pi$ is indeed a connection.

Finally, if $s$ is a smooth section of $\bC TU$ over some neighborhood $U\subset N$ of $x$, then 
$$
\nabla_vs'=\Pi s'_*v=[w,s]'(x)
$$ 
by (5.2). This is linear in $s'$, and $\Pi$ is a linear connection.
\end{proof}

We still need to prove
\begin{lem}   
Suppose $V$ is a Banach space, $N$ a Banach manifold, $E\to N$ a smooth Banach bundle, $F\subset E$ a subbundle, 
$q:E\to E/F$ the quotient map, and
 $\theta: E/F\to V$ an arbitrary map. If $\psi=\theta\circ q$ is smooth, then so is $\theta$.
\end{lem} 

The proof of Lemma 5.5 uses this with  $\psi(v,a)=\big(x,\tilde\theta(v,a)\big)'$, $v\in \oP_x$, $a\in A$. 

\begin{proof}
We can assume $N$ is an open subset of a Banach space $X$, and $E=N\times A\to N$, $F=N\times B\to N$ with Banach 
spaces $A\supset B$. Let $Z=A/B$. We extend 
$q$ to a linear map $X\times A\to X\times Z$, which we still denote $q$. By Michael's selection theorem \cite{Mi} the quotient 
map $A\to Z$
has a continuous right inverse, and this gives rise to a continuous right inverse $r:X\times Z\to X \times A$ of $q$. We will 
show by induction
that iterated directional derivatives  
$\eta_1\dots\eta_\nu{\theta(y)}$ of all orders $\nu=0,1,\ldots$ and $\eta_j\in X\times Z$, $y\in N\times Z$ exist, and
\begin{equation}  
(\eta_1\dots\eta_\nu \theta)(qz)=\big((r\eta_1)\dots (r\eta_\nu)\psi\big)(z), \quad z\in N\times A.
\end{equation} 

This we know for $\nu=0$. Suppose it is true for some $\nu\ge 0$. If  $\eta_0,\dots, \eta_\nu\in X\times Z$,
$z\in N\times A$, 
\[
\frac{d}{dt}\Big |_{t=0}(\eta_1\dots\eta_\nu \theta)\big(qz+t\eta_0\big)=
\frac{d}{dt}\Big |_{t=0}\big((r\eta_1)\dots (r\eta_\nu)\psi\big)(z+tr\eta_0)\big),\quad z\in N\times A, 
\]
exists and equals $\big((r\eta_0)\dots (r\eta_\nu)\psi\big)(z)$. Thus $\theta$ indeed has directional derivatives of all 
orders and (5.11) holds. By (5.11) the directional derivatives $(\eta_1\dots\eta_\nu \theta)(y)$ depend continuously on 
$\eta_1,\dots,\eta_\nu, y$, whence $\theta$ is smooth as claimed.
\end{proof}

\section{Variation of involutive maps}

We start with a single involutive map $f:(M,Q)\to (N,P)$. The Bott connection $\nabla$ on $P'\to N$ pulls back to a 
$\oQ$--connection $\nabla^f$ on the induced bundle $f^*P'\to (M,Q)$. Given a smooth section $s$ of $\bC TM$ over 
some open $U\subset M$, its pushforward $f_*s$ is a vector field along $f$, i.e., a section of $f^*\bC TN| U$. Since $f_*s$ 
mod $\oP$ depends only on $s$ mod $\oQ$, there is an induced map $Q'\to f^*P'$, also denoted $f_*$. 

\begin{lem}  
If $w,s$ are sections of $\oQ$, $\bC TM$ over some open $U\subset M$, then
\begin{equation}  
\nabla^f_w f_*s'=f_*[w,s]'.
\end{equation} 
\end{lem}   

When $f=\id_{(N,P)}$, (6.1) reduces to (5.8). The proof depends on certain operations with involutive 
manifolds and partial connections. Given involutive manifolds $(N_j, P_j)$ for $j=1,2$, their product
$(N_1,P_1)\times(N_2,P_2)$ is the involutive manifold $(N_1\times N_2, P_1\times P_2)$ with 
$(P_1\times P_2)_{(x,y)}=
P_{1x}\oplus P_{2y}\subset (\bC T_x N_1)\oplus(\bC T_yN_2)\approx \bC T_{(x,y)} N$. 
Suppose $E_j\to (N_j, P_j)$ for $j=1,2$ are complex Banach bundles 
with $\oP_j$--connections $\Pi_j$ and corresponding covariant differentiations $\nabla_j$. The product
Banach  bundle
\[E=E_1\times E_2\to (N_1, P_1)\times (N_2, P_2)\]
has a natural $\oP_1\times \oP_2$ connection $\Pi=\Pi_1\boxplus \Pi_2$ given by $\Pi(\xi,\eta)=\big(\Pi_1(\xi),\Pi_2(\eta)\big)\in E$.
The corresponding covariant differentiation $\nabla_1\boxplus\nabla_2$ is given by
\begin{equation}
(\nabla_1 \boxplus\nabla_2)_{(v_1, v_2)}s=\nabla_{1v_1}s_1\oplus \nabla_{2v_2}s_2, \quad \text{if } s(x,y)=\big(s_1(x), s_2(y)\big).
\end{equation} 

Suppose next that $\pi_j:E_j\to (N, P)$ for $j=1,2$ are complex Banach 
bundles over the same involutive manifold. $\oP$--connections $\Pi_j$ on $E_j$ induce a $\oP$--connection $\Pi$ on
$\pi:E=E_1\oplus E_2\to N$, 
$$
\Pi(\xi,\eta)=\Pi_1(\xi)\oplus\Pi_2(\eta)\in E,\qquad\text{for}\quad (\xi,\eta)\in\pi_{1*}^{-1}\oP\oplus\pi_{2*}^{-1}\oP
\approx\pi_*^{-1}\oP.
$$
The corresponding covariant derivatives $\nabla_1,\nabla_2,\nabla_1\oplus\nabla_2$ are related by
\[
(\nabla_1\oplus\nabla_2)_v(s_1\oplus s_2)=\nabla_{1v} s_1\oplus \nabla_{2v} s_2.
\]
\begin{proof}[Proof of Lemma 6.1] 
Regardless of involutive structures, if $w,s\in \Gamma(\bC TM)$ and $v,t\in \Gamma(\bC TN)$ satisfy
\begin{gather}
f_*w=f^*v,\quad f_*s=f^*t, \qquad\text{then }\\
f_*[w,s]=f^*[v,t].
\end{gather} 
Indeed, if $U\subset M$ is open and $\varphi$ is a smooth function on a neighborhood of $f(U)$, 
\[wsf^*\varphi=w(f_*s)\varphi=w(f^*t)\varphi=wf^*(t\varphi)=(f_*w)(t\varphi)=f^*(vt\varphi)\]
and similarily for $swf^*\varphi$. Hence
\[
(f_*[w,s])\varphi=[w,s]f^*\varphi=f^*([v,t]\varphi)=(f^*[v,t])\varphi,
\]
and (6.4) follows.

In the set up of Lemma 6.1 suppose first that $f:M\to N$ is a direct embedding. At the price of shrinking we can find 
$v\in\Gamma(\oP)$, $t\in\Gamma(\bC TN)$ such that (6.3) holds. Since $\nabla_v t'=[v,t]'$, cf. (5.8), by (5.5) and (6.4) at any $x\in N$
\[\nabla_{w(x)}^ff_*s'=\nabla^f_{w(x)}f^*t'=f^*\nabla_{v(f(x))}t'=f^*[v,t]'(x)=f_*[w,s]'(x).\]

Second, if $f$ is arbitrary, consider the involutive map 
\[g:M\ni x\mapsto \big(x,f(x)\big)\in M\times N,\]
a direct embedding. If the Bott connections on $Q'\to (M,Q), P'\to (N,P)$ are denoted $\nabla^M$ and $\nabla$, the Bott connection on $(M,Q)\times(N,P)$ is $\nabla^M\boxplus\nabla$. Its pullback by $g$ is $\nabla^g=\nabla^M\oplus\nabla^f$. By what we have proved, 
\[\nabla^M_w s'\oplus\nabla^f_w s'=\nabla^g_w g_*s'=g_*[w,s]'=[w,s]'\oplus f_*[w,s]'.\]
Hence (6.1) follows.
\end{proof}   

\begin{thm} 
Let $(M,Q), (N,P)$ be involutive manifolds, $f:\bR\times M\to N$ a smooth map. If $f_\lam=f(\lam,\cdot):(M,Q)\to (N,P)$ is involutive for all $\lam\in\bR$, then $\partial f/\partial\lam\in\Gamma(f_\lam^* TN)$ satisfies
\[
\nabla_v^{f_\lam}(\frac{\partial f_\lam}{\partial\lam})'=0 \qquad\text{for all}\quad v\in\oQ.
\]
\end{thm}  

\begin{proof}
We view $\bR$ as an involutive manifold with its involutive structure the rank zero subbundle of $\bC T\bR$. Its product with $(M,Q)$ is an involutive manifold $(\bR\times M, R)$; 
\[f:(\bR\times M, R)\to (N,P)\]
is involutive, as are the projections $\text{pr}_\bR:\bR\times M\to \bR$ and $\text{pr}_M:\bR\times M\to M$. With $\lam$ 
the coordinate on $\bR$, we lift $\partial/\partial\lam\in \Gamma(T\bR)$ to the vector field 
$\partial_\bR=(\partial/\partial\lam,0)$ 
on $\bR\times M$. Similarly, given a section $w$ of $\oQ$ 
over some open $U\subset M$, we lift it to the section $\omega=(0,w)$ of $\oR|\bR\times U$. Thus 
$[\omega,\partial_\bR]=0$. By Lemma 6.1
\[\nabla^f_\omega(f_*\partial_\bR)'=f_*[\omega,\partial_\bR]'=0.\]
Restricting to the fibers $\{\lam\}\times M$ gives the theorem. In greater detail, for fixed 
$\lam\in\bR$ let $\var:M\ni x\mapsto (\lam,x)\in \bR\times M$.  Thus $f_\lam=f\circ \var$ and 
$\partial f_\lam/ \partial\lam=\var^*f_*\partial_\bR$; furthermore pullback by $f_\lam$ is the composition of pullbacks by
$f$ and $\var$. Applying (5.5) with $\nabla=\nabla^f$ and $g=\var$, 
\[\nabla^{f_\lam}_v(\frac{\partial f_\lam}{\partial\lam})'=\var^*\nabla^{f}_{\var_* v}(f_*\partial_\bR)'=0, \quad v\in\oQ.\]
\end{proof}

\section{Uniqueness of adapted involutive structures}

Suppose a Lie monoid $G$ acts on the right on a smooth manifold $N$, and is also endowed with an involutive structure $S$. If $N$ admits an 
involutive structure adapted to the action of $(G,S)$, is this structure unique?

In general one cannot expect a positive answer. For example, if $S$ is of rank 0, any $(N,P)$ is adapted to it. 
More to the point, if the orbits of $G$ foliate $N$, then for $P$ to be adapted may constrain what it can 
be in the direction of the leaves,  but places little restriction on what it can be in transverse directions.
Nonetheless, for certain actions on spaces of geodesics uniqueness has already been proved under suitable
assumptions in \cite{LSz1, LSz2, Sz4, Sz5}; there the orbits did not foliate. In this section we prove two abstract 
uniqueness theorems that apply to actions with rather special properties, generalizing the
uniqueness theorems above and also applying to spaces of harmonic maps. Both theorems involve assumptions on the action,
on $S$, and on $P$ as well; but the most important assumption is that the action of $G$ extends to  an action
of a certain space $\hg$, and $g\in\hg\setminus G$ act by degenerate maps. When $N$ is a space of geodesics and $G$ acts by affine
reparametrizations $t\mapsto a+bt$, $\hg$ will contain non-invertible maps $t\mapsto a$; when $N$ is a space of harmonic maps from a convex
domain $D\subset\bC$, and $G$ is a submonoid of $\psl_2(\bC)$, $\hg$ again will contain constant self maps $D\to D$. Both situations are
covered by the following

\begin{defn}  
A degeneration of a right action of a Lie monoid $G$ on a manifold $N$ is given by a topological space $\hg\supset G$ and a continuation of the action to a continuous map 
\begin{equation}  
N\times\hg\ni (x,g)\mapsto xg\in N.
\end{equation}   
Writing $xg=A_gx=\Omega_xg$, we require that $A_g$ be differentiable for each $g\in\hg$;
$$
\hg\ni g\mapsto A_{g*}|\bC T_xN\in\Hom(\bC T_xN,\bC T_{xg}N),
$$
as a section of the Banach bundle $\Hom(\bC T_xN,\Omega^*_x\bC TN)$, be continuous; and
for any $g\in\hg\setminus G$ the map $A_g$ be a submersion on a fixed submanifold $M\subset N$.
\end{defn}  

The assumption that we will make on the involutive structure $S$ has to do with unique continuation. Quite generally, let 
$(M,Q)$ be an involutive manifold. A smooth function $f:(M,Q)\to (\bC, T^{10}\bC)$ is involutive if and only if $vf=0$ for all 
$v\in\oQ$. Indeed, decomposing the holomorphic coordinate on $\bC$ as $z=x+iy$, and writing $f=f_1+if_2$ accordingly, for 
any $v\in \bC TM$\[f_*v=(vf_1)\frac{\partial}{\partial x}+(vf_2)\frac{\partial}{\partial y}=
(vf)\frac{\partial}{\partial z}+(v\of)\frac{\partial}{\partial \oz},\]
whence the claim follows. In Treves' book \cite{Tr} such functions are called ``solutions'' of the involutive structure. The same 
characterization holds for involutive maps $f:M\to V$ into a complex Banach space, when the involutive structure of $V$ is 
$T^{10}V$. 

In the next definition $(G,S)$ will be an involutive Lie monoid, a subspace of a topological space $\hg$. The Lie monoid is also a subset of a manifold $X$ (cf. Definition 2.1), and $\In G$ will refer to interior relative to $X$.

\begin{defn}  
Suppose $(G,S)$ is an involutive Lie monoid, $G$ a subspace of a topological space $\hat G$. We say that an open 
$U\subset\hat G$ is a set of uniqueness if a continuous function $f:U\to\bC$ that is involutive on $U\cap \In G$ and vanishes on 
$U\setminus G$ must be identically 0.   
\end{defn}  

For example, if $X\subset\bC$ is open, $G\subset X$,  $\In G$ is a Jordan domain, $S=T^{10}G$, and $\hg\setminus G$ is
an arc on the boundary of $G$, then any connected open $U\subset\hg$ that intersects $\hg\setminus G$ is of 
uniqueness.---In the definition we could replace $\bC$ valued functions by Banach
space valued functions, because Banach valued functions can be reduced to $\bC$ valued functions by composing them with linear forms.

\begin{thm}  
Suppose a Lie monoid $G$ acts on the right on a smooth manifold $N$ and the action has a degeneration 
\begin{equation}  
N\times\hg\ni(x,g)\mapsto xg=A_g x=\Omega_xg \in N.
\end{equation}   

Let $S$ be any involutive structure on $G$. Fix $x\in N$ and a closed subspace $\Xi\subset \bC T_xN$. There is at most one subspace $P_x\subset \bC T_x N$ which can be continued to an involutive structure $P\subset \bC TN$ adapted to the action of $(G, S)$ that has the following property: There is an open set $U\subset\hg$ of uniqueness containing $e\in G$ such that
\[
\bC T_xN=\Xi\oplus( \text{Ker }{A_{g*}})_x\quad \text{for } g\in U\setminus G,
\]
and the composition
\[\Xi \xrightarrow {A_{g*}}\bC T_{xg}N\xrightarrow {\text{pr}} P'_{xg}\]
is an isomorphism for $g\in U$. (The second map here is projection.)
\end{thm}   

As we shall see, in certain situations---e.g. when $N$ is a space of geodesics, see Theorem 7.5---$\Xi\subset\bC T_xN$ as in the theorem can be found for a dense set of $x\in N$, and for the uniqueness of $P_x$ it is not even necessary to know what $\Xi$ is. In such a case the adapted $P$ is simply determined by the action of $(G,S)$. 

In what follows, we write
\[
A'(g)=\text{pr}\circ A_{g*}: \bC T_xN\to P'.
\]

\begin{proof} [Proof of Theorem 7.3] 
We will reconstruct $P_x$ from the extended action (7.2). Let $M\subset N$ be as in Definition 7.1 and $P$ as in the 
theorem. Since $A_{g*}|\Xi$ is an isomorphism on $\bC T_{xg} M$ when $g\in U\setminus G$, 
$\varphi(g)=(A_{g*}|\Xi)^{-1}A_{g*}$ defines a map $\varphi :U\setminus G\to \Hom(\bC T_xN,\Xi)$ such that
\begin{equation}  
A_{g*}\zeta=A_{g*}\varphi(g)\zeta, \quad g\in U\setminus G,\,\,  \zeta\in \bC T_x N.
\end{equation}  
Note that $\varphi$ is determined by the action (7.2) and $\Xi$, and of course the choice of $U$. 

Similarly, $\Phi(g)=(A'(g)|\Xi)^{-1}A'(g)$ defines a continuous map $\Phi: U\to \Hom(\bC T_xN, \Xi)$, smooth over $U\cap \In G$ such that
\begin{equation} 
A'(g)\zeta=A'(g)\Phi(g)\zeta, \quad g\in U,\,\, \zeta\in \bC T_xN.
\end{equation}  
Unlike $\varphi$, the action itself does not immediately determine $\Phi$, since its definition involves $A'(g)$, which depends on the unknown
$P$. However, by (7.3)  $A'(g)\zeta=A'(g)\varphi(g)\zeta$ for $g\in U\setminus G$. As $\Xi\cap \Ker \,A'(g)=(0)$, together with (7.4) this 
implies
\begin{equation} 
\Phi(g)=\varphi(g) \quad \text{for }  g\in U\setminus G.
\end{equation}  
Furthermore,
$$
\Ker\,\Phi(e)=\Ker\, A'(e)=\oP_x.
$$

Next we show that $\Phi\zeta:U\cap \In G\to \Xi$ is involutive for all $\zeta\in \bC T_x N$ (which is the same as the involutivity of $\Phi$ itself). Let $\bR\ni\lam\mapsto x_\lam\in N$ be a smooth curve such that 
$\zeta=dx_\lam/d\lam \big |_{\lam=0}$, and let $f_\lam=\Omega_{x_\lam}|U\cap\In G$. Note that 
$f_\lam:(U\cap\In G,S)\to (N,P)$ is involutive (because $P$ is adapted), $f_0(g)=A_gx$, and 
$\partial f_\lam(g)/\partial\lam\big |_{\lam=0}=A_{g*}\zeta$.  By Theorem 6.2 
$\nabla^{f_\lam}(\partial f_\lam/\partial\lam)'=0$; setting $\lam=0$ we see that the section
\[
U\cap\In G\ni g\mapsto A'(g)\zeta\in P'_{xg}
\]
of the bundle $f_0^*P'$ is annihilated by $\nabla^{f_0}$. This also applies with $\zeta$ replaced by $\Phi(g_0)\zeta$, 
where $g_0$ is fixed, i.e., 
\[s(g)=A'(g)\Phi(g)\zeta=A'(g)\zeta \qquad \text{and }\qquad t(g)=A'(g)\Phi(g_0)\zeta\]
satisfy $\nabla^{f_0}s=\nabla^{f_0}t=0$, or $\nabla^{f_0}(s-t)=0$. Since $s-t=0$ at $g_0$, the expression (5.4)
of a partial connection in a local trivialization therefore gives
\[0=
A'(g_0)v(\Phi\zeta-\Phi(g_0)\zeta)=A'(g_0)v(\Phi\zeta), \quad v\in \oS_{g_0}.
\]
Both $\Phi\zeta$ and $v(\Phi\zeta)$ take values in $\Xi$, and the restriction of $A'(g_0)$ to $\Xi$ is an isomorphism. Hence $v(\Phi\zeta)=0$ for $v\in\oS_{g_0}$, and $\Phi\zeta$ is indeed involutive.

As $U$ is a set of uniqueness, this and (7.5) imply that $\Phi$ is uniquely determined by $\varphi$. Thus
$\oP_x=\Ker\,\Phi(e)$  can be reconstructed,
through $\varphi$, from $\Xi$, $(G,S)$, the action (7.2), and $U$. Of course, the reconstructed $\oP_x$ is independent of the choice of $\Xi,U$ satisfying the assumptions. Therefore $P_x$ is indeed unique.
\end{proof}

At first glance the assumptions of the theorem are problematic, since they involve a seemingly artificial property of the unknown involutive 
structure $P$. In fact, the relevant assumption can be derived from crude properties of the $G$ action and of $P$ (its co-rank), plus unique
continuation results in involutive structures, as we next illustrate. In various situations connected open neighborhoods of continuous curves
in $\hg$ are of uniqueness, and the isomorphism  condition in Theorem 7.3 is satisfied by generic $\Xi$ if the neighborhood is sufficiently
thin. This is what we need:

\begin{prop} 
Let $V$ be a complex Banach space, $E\to (0,1)$ a smooth complex Banach bundle, $F\subset E$ a smooth subbundle, and $\alpha_t: V\to E_t$ a smooth family of continuous monomorphisms, $0<t<1$. Given $m\in\bN$, if
\begin{equation} 
\dim \alpha_t(V)/(F\cap \alpha_t(V))\ge m\quad \text{for all } t\in(0,1),
\end{equation}  
then for a generic $Z$ in the Grassmannian ${\rm{Gr}}={\rm{Gr}} (V,m)$ of $m$ dimensional subspaces of $V$
\[\alpha_t(Z)\cap F_t=(0) \quad \text{for all } t\in (0,1).\]
\end{prop} 
The Grassmannian in question is the image of the space GL$(V)$ of invertible endomorphisms of $V$ under the map that associates with $A\in \text{GL}(V)$ the image $AW$ of a fixed $m$ dimensional subspace $W\subset V$. The norm topology on $\text{GL}(V)$ induces a completely metrizable topology on Gr, and generic means `in a dense $G_\delta$ set'.

\begin{proof}
It suffices to show that given $t\in(0,1)$, there is an $\var>0$ such that
\[\cZ_\var=\{Z\in\text{Gr}:\alpha_\tau(Z)\cap F_\tau=(0) \quad\text{if } |\tau-t|\le\var\}\subset \text{Gr}\]
is open and dense. This we will show under the sole assumption that (7.6) holds for the $t$ in question. For this reason we can assume, at the price of restricting to a neighborhood of $t$ and enlarging $F$, that $E=F\oplus T$ with a trivial bundle $T$ of rank $m$.

Let $\varrho: F\oplus T\to T$ denote the projection, so that $Z\in\cZ_\var$ if and only if $\varrho\alpha_\tau: Z\to T_\tau$ 
is an isomorphism,
or $\text{det}(\varrho\alpha_\tau| Z)\ne 0$, for $|\tau-t|\le\var$. The determinant is computed in a fixed basis of $Z$ and a fixed frame of
$T$. When $\var<t,1-t$, the latter condition defines an open set $\cZ_\var\subset \text{Gr}$. By (7.6) there is a
$Z_0\in \text{Gr}$ such that $F_t\cap\alpha_t(Z_0)=(0)$, i.e., det$(\varrho\alpha_t | Z_0)\ne 0$. Choose $\var>0$ so small that
\begin{equation} 
\text{det}(\varrho\alpha_\tau | Z_0)\ne 0\quad \text{for } |\tau-t|\le\var.
\end{equation}  
We claim that this implies that $\cZ_\var$ is dense in Gr.

Indeed, let $W\in\text{Gr}$. With a suitable linear $L:Z_0\to V$
\[
W=\{\zeta+L\zeta: \zeta\in Z_0\}. 
\]
The function
\[p_\tau(s)=\det \varrho\alpha_\tau(s\text{Id}_{Z_0} +L)=\sum\limits^m_0 a_j(\tau){s^j}, \quad |\tau-t|\le\var,\quad s\in\bC,\]
is a polynomial in $s$ of degree exactly $m$ (by (7.7)), whose coefficients are smooth functions of $\tau$. Note that
\begin{equation}
\begin{aligned} 
\{s\in\bC &: p_\tau(s)=0\quad \text{for some } \tau\}\subset\\ &\bigcup\limits^m_{j=0}\{s\in \bC \text{ is a simple zero of  } \partial^j  {p_\tau}(s)/\partial s^j \text{ for some } \tau\}.
\end{aligned}  
\end{equation}
By the implicit function theorem, each set in this union is a countable union of smooth plane curves. It follows that the set 
(7.8) has zero planar measure, and arbitrarily close to 1 there are $s$ such that $p_\tau(s)\ne 0$ when $|\tau-t|\le\var$. 
But then the subspaces
\[
W_s=\{s\zeta+L\zeta:\zeta\in Z_0\}
\]
are in $\cZ_\var$ and arbitrarily close to $W$. We conclude that $\cZ_\var$ is indeed open and dense.
\end{proof}

Combining Theorem 7.3 and Proposition 7.4 we can prove uniqueness of adapted involutive structures on spaces of geodesics:

\begin{thm}  
Consider a smooth manifold $M$ of dimension $m<\infty$ endowed with a smooth linear connection $\nabla$ on its tangent bundle and, given $r\in(0,\infty)$, the $2m$ dimensional manifold $N_0$ of its geodesics $x:[-r,r]\to M$. By associating with $p\in M$ the constant geodesic $x\equiv p$, we can view $M$ as a submanifold of $N_0$. The Lie monoid $G$ of affine maps
\[ [-r,r]\ni t\mapsto a+bt\in [-r,r], \quad a,b\in \bR,\, b>0, \, |a|+br\le r,\]
acts on $N_0$ by reparametrizations. For $c\in\bC$ let $S(c)$ be (the restriction of)  the involutive structure on $G$
of Example 3.2. Suppose $c$ is not real or else is in $(-r, r)$. 

If an open $G$ invariant  neighborhood $N$ of $M\subset N_0$ admits an involutive structure $P$ adapted to the action of $(G, S(c))$ and $rk\,P\le m$, then $P$ is unique.
\end{thm}  

This theorem generalizes \cite[Theorem 2a]{LSz2} in several ways; in particular, it allows non-Riemannian geodesics and 
incomplete connections. (In the context of adapted complex structures, \cite{Sz4} already dealt with this generality
and proved uniqueness.) The 
proof in \cite{LSz2} is by reduction to \cite[Theorem 4.2]{LSz1}. Sz\H{o}ke has lately observed that this reduction works 
only if $P$ is assumed to be a complex structure and not only a polarization (an involutive structure of rank $m$). The 
proof to follow works even if $P$ is not a complex structure. Nonetheless, its main component is Theorem 7.3, whose 
proof is but a variant of the proof of \cite[Theorem 4.2]{LSz1}.

\begin{proof}
We take $\hg$ to consist of affine transformations $g_{ab}: t\mapsto a+bt$ with $b\ge 0$ and $|a|+br\le r$, that we view as a subspace of the $ab$ plane. Then $S(c)$ of Example 3.2 extends to an involutive structure on $\hg$, as the span over $\bC$
\begin{equation} 
\langle\oc\partial_a-\partial_b\rangle=S(c).
\end{equation}  
Reparametrizations by $g\in\hg$ define a degeneration of the action of $G$. We have 
\[
\hg\setminus G=\{g_{a0}: -r\le a\le r\}.
\]
We will check that for every nonconstant geodesic $x\in N\setminus M$ the assumption of Theorem 7.3 on $P$ is satisfied by a generic $\Xi\in\text{Gr}=\text{Gr}(\bC T_xN,m)$. 

If $c\in(-r,r)$, let $\gamma=g_{c0}\in \hg\setminus G$; if $c\in \bC\setminus \bR$, let $\gamma\in \hg\setminus G$ be 
arbitrary. Connect $e=g_{01}$ and $\gamma$ with a line segment $\Gamma\subset \hg$ (line segment when viewed as a 
subset of the plane). No matter what generic $\Xi\in \text{Gr}$ we choose,
any sufficiently narrow parallel strip around the line of $\Gamma$, intersected with $\hg$, will do as $U$ of Theorem 7.3.  
 
To show this, first we compute $P|M$ as follows. For any $x\in N$, $T_x N$ consists of Jacobi fields along $x$. When $x\equiv u\in M$ is a 
constant geodesic, the solutions of Jacobi's equation are affine maps into $T_uM$:
 \[
 T_uN=\{\theta: [-r,r]\to T_uM \text{ is affine}\}.
 \]
There is an analogous description of $\bC T_uN$. For any $x\in N$, from $(\Omega_xg_{ab})(t)=x(a+bt)$ we can compute $\Omega_{x*}$. When $b=0$,
we find
\[\Omega_{x*}(\oc\partial_a-\partial_b)\big |_{g_{a0}} (t)=\theta(t)=(\oc-t)\dot x(a).\]
Since $P$ is adapted, $\theta\in P$, see (7.9). Fix $x(a)=u\in M$. As $\dot x(a)\in T_uM$ varies, the corresponding $\theta$ sweep out an 
$m$ dimensional totally real subspace of $\bC T_uN$, whose complex linear span  is $m$ complex dimensional and 
contained in $P$. But
$\text{rk}\, P\le m$,  so this span is simply $P_u$. Since the tangent space to $M\subset N$ consists of constant maps $\theta:[-r,r]\to T_uM$, the
subspaces $P_u$ and $\bC T_uM$ are transverse in $\bC T_uN$. As $A_{g_{a0}}$ is a submersion on $M$, for any $\Xi\in \text{Gr}$ complementary
to $\bC T_xN\cap \text{Ker} A'{(g_{a0})}$
\begin{equation}   
\Xi\xrightarrow{A_{g*}}\bC T_{xg} N\xrightarrow{\text{pr}}P'_{xg}
\end{equation}  
is an isomorphism for $g=g_{a0}$. Using this with $g_{a0}=\gamma$, Proposition 7.4 then implies that with a generic $\Xi\in \text{Gr}$ (7.10) is an isomorphism  for all $g\in\Gamma$, and by continuity, for $g\in\hg$ in a sufficiently narrow  strip $U$ about the line of $\Gamma$. 

At the same time, this $U$ will be of uniqueness. When $c\notin\bR$, the identification $\hg\ni g_{ab}\leftrightarrow a+bc\in \bC$ turns involutive functions on $\big(G,S(c)\big)$ into holomorphic functions in $\bC$. Uniqueness follows, since for any Jordan domain $D\subset\bC$ any continuous function on $\oD$ that is holomorphic in $D$ and vanishes on a nontrivial arc of $\partial D$ must be identically 0. When $c\in(-r,r)$, involutive functions are constant along lines parallel to $\Gamma$. Hence if they vanish on $U\setminus G$, they vanish identically. 

We have thus verified the assumptions of Theorem 7.3. By that theorem therefore the choice of $\Xi\in\text{Gr}$ determines $P_x$. Although without the a priori knowledge of $P$ we do not know which $\Xi\in\text{Gr}$ satisfies the relevant assumption, it suffices that generic $\Xi$ determine the same $P_x$. In this way, $P_x$ is determined solely by the action of $(G,S)$, without reference to $\Xi$. As this holds for a dense set of $x\in N$, $P$ itself is unique.
\end{proof}

The assumptions of Theorem 7.3 imply $\Xi\approx T_uM$ for any $u\in A_gN$, $g\in U\setminus G$, and also $\Xi\approx P'_x$. For example, if $M$ is finite dimensional, the theorem applies only for involutive structures $P$ of co-rank equal to dim$\, M$. Next we prove an analog in which the co-rank of $P$ can be greater than dim$\,M$.

\begin{thm}   
Suppose a Lie monoid $G$ acts on the right on a smooth manifold $N$ and the action has a degeneration 
\[N\times\hg\ni(x,g)\mapsto xg=A_gx=\Omega_xg\in N.\]
Let $S$ be an involutive structure on $G$, $x\in N$, and fix closed subspaces $\Xi$, $H\subset \bC T_xN$. There is at most one subspace $P_x\subset \bC T_x N$ which can be continued to an involutive structure $P\subset \bC TN$ adapted to the action of $(G, S)$ that satisfies the following. With some open set $U\subset\hg$ of uniqueness containing $e\in G$, in which $U\cap G$ is dense, with some $\chi:U\cap G\to(0,\infty)$ and $\beta(g)\in\Hom(\bC T_xN,A_{g*}\Xi)$ for $g\in U\cap G$,
\begin{equation}
B(g)=\chi(g)\big(A_{g*}-\beta(g)\big):\bC T_xN\to\bC T_{xg}N
\end{equation}
has a limit in the norm topology on operators as $g\to g_0$ for any $g_0\in U\setminus G$. Denoting the limit $B(g_0)$,
\begin{gather}
\bC T_xN=H\oplus B(g)^{-1}(A_{g*}\Xi)\qquad{\text{for any $g\in U\setminus G;\quad$ the operators}}\\
A'(g):\bC T_xN\xrightarrow{A_{g*}}\bC T_{xg}N\to P'_{xg},\qquad g\in U,
\end{gather}
restricted to $\Xi$, define an isomorphism between the trivial bundle $U\times\Xi\to U$ and a subbundle 
$E\subset\Omega_x^*P'|U$; and the operators
\begin{equation}
B'(g):\bC T_xN\xrightarrow{B(g)}\bC T_{xg}N\to\bC T_xN / (\oP_{xg}+A_{g*}\Xi)\approx P'_{xg}/A'(g)\Xi=P'_{xg}/E_g,
\end{equation}
restricted to $H$, define an isomorphism between the trivial bundle $U\times H\to U$ and $(\Omega^*_xP'|U)/E$. The two isomorphisms are topological, and smooth over $U\cap\In G$.
\end{thm}    

In the compositions in (7.13), (7.14) the second maps are canonical projections.---Theorem 7.3 corresponds to the special case $H=(0)$ 
$\chi\equiv 1$, $\beta\equiv 0$. However, when $H$ is positive dimensional, typically $\chi(g)\to \infty$ when $g\to U\setminus G$. As with Theorem 7.3, in certain
situations crude assumptions on $P$ already guarantee that $\Xi, H, U,\chi$, and $\beta$ with the required properties can be found for a dense set
of $x\in N$; moreover $P_x$  can be reconstructed uniquely without even knowing what $\Xi$ and $H$ are. This is the case
with certain spaces of harmonic maps. However, the application of Theorems 7.6 and 7.3 to spaces of harmonic maps will have to wait until a
subsequent publication. 

\begin{proof}
Again, we will reconstruct $\oP_x$ from the action and data $\Xi, H$. Let $P,U$ be as in the theorem.
For $g\in U\setminus G$ denote by $\psi(g)\in\Hom (\bC T_xN, H)$ projection in the decomposition  (7.12), so that
\begin{equation}   
B(g)=B(g)\psi(g)\quad\text {mod } A_{g*}\Xi,\qquad \text{hence}\qquad B'(g)=B'(g)\psi(g).
\end{equation}  
Another consequence is that there are unique linear maps $\varphi(g):\bC T_xN\to\Xi$ such that
\begin{equation} 
B(g)=A_{g*}\varphi(g)+B(g)\psi(g),\qquad g\in U\setminus G;
\end{equation}
uniqueness follows since $A'(g)$, hence $A_{g*}$ are injective on $\Xi$.

At the same time $\Psi(g)=\big(B'(g)|H\big)^{-1}B'(g)\in\Hom(\bC T_xN,H)$ satisfies
\begin{equation}   
B'(g)=B'(g){\Psi}(g), \quad g\in U.
\end{equation}  
Note that $\Psi$ is continuous on $U$ and smooth over $U\cap\In G$. (7.17) means that 
$$
A'(g)\big(\zeta-\Psi(g)\zeta\big)\in A'(g)\Xi,\qquad \zeta\in\bC T_xN,\quad g\in U.
$$
Hence $\Phi(g)=\big(A'(g)|\Xi\big)^{-1}A'(g)\big(\Id_{\bC T_xN}-\Psi(g)\big)\in\Hom(\bC T_xN,\Xi)$ satisfies
\begin{equation}   
A'(g)(\zeta-{\Psi}(g)\zeta)=A'(g)\Phi(g)\zeta.
\end{equation}  
Using Theorem 6.2 as in the proof of Theorem 7.3 we differentiate (7.18) along a vector $v\in \oS_g$, $g\in U\cap \In G$ to obtain
\begin{equation}   
A'(g)v({\Psi}\zeta)=-A'(g)v(\Phi\zeta)\in A'(g)\Xi,
\end{equation}  
or $B'(g)v(\Psi\zeta)=0$. Since $B'(g)| H$ is injective, $v({\Psi}\zeta)=0$, and since $A'(g)|\Xi$ is injective, $v(\Phi\zeta)=0$ by
(7.19). To sum up, $\Phi\zeta$ and ${\Psi}\zeta$ are involutive. Comparing (7.15), (7.17), respectively (7.16), (7.18), we see 
$\Psi=\psi$ and $\Phi=\varphi$ on $U\setminus G$. Therefore $\Phi,\Psi$ are uniquely determined by 
$\varphi,\psi$, which in turn
are determined by the input in Theorem 7.6. Letting $g=e$ in (7.18) we obtain
\[
\zeta-{\Phi}(e)\zeta-\Psi(e)\zeta\in\oP_x \qquad \zeta\in \bC T_x N,
\]
hence $\Ker\big(\Phi(e)+\Psi(e)\big)\subset \oP_x$. In fact, the two are equal. Indeed, let
$\zeta\in\oP_x$. Clearly 
$A'(e)\zeta=\zeta'=0$, and similarly $B'(e)\zeta=0$ by (7.11), (7.13), (7.14). Therefore by (7.17) $B'(e)\Psi(e)\zeta=0$ and 
${\Psi}(e)\zeta=0$. But then (7.18) implies $A'(e)\Phi(e)\zeta=0$ and $\Phi(e)\zeta=0$. Thus 
$\oP_x=\Ker\big(\Phi(e)-\Psi(e)\big)$ is uniquely determined.
\end{proof}

\

\end{document}